\documentclass[12pt]{amsart}

\usepackage{fullpage}
\usepackage{amsmath}
\usepackage{amsfonts}
\usepackage{amssymb}
\usepackage{amsthm}
\usepackage[all,knot,poly]{xy}
\usepackage{xypic}
\usepackage{hyperref}
\hypersetup{
    colorlinks=true,
    linkcolor=blue,
    filecolor=magenta,      
    urlcolor=cyan,
}
 
\input xy
\xyoption{all}

\theoremstyle{plain}

\newtheorem{theorem}{Theorem}[section]
\newtheorem{lemma}[theorem]{Lemma}
\newtheorem{proposition}[theorem]{Proposition}

\newtheorem{remark}[theorem]{Remark}
\newtheorem{definition}[theorem]{Definition}
\newtheorem{example}[theorem]{Example}

\usepackage{color}

\title{{Bidirectional Sequential Motion Planning}}
\author{Enrique Torres-Giese}
\address{Trinity Western University, Langley BC, V2Y 1Y1 , Canada.}
\email{enrique.torresgiese@twu.ca}

\begin{document}
\begin{abstract} 
We define a simpler notion of symmetric topological complexity more ad hoc to the 
motion planning problem which was the original motivation for the definition of topological complexity. This is a homotopy invariant that we call bidirectional 
topological complexity. We prove properties of this invariant and show specific 
instances for which the symmetrized topological complexity can be relaxed to the
bidirectional setting. This approach allows us to estimate higher values of 
symmetrized topological complexities.  
\end{abstract}
\maketitle

\section{Introduction}
The topological complexity of a space $X$, denoted by $\mathrm{TC}_2(X)$, 
is a homotopy invariant that intends to measure the difficulty of 
the motion planning problem in the space $X$   (see~\cite{farber}).
More precisely, it is the smallest number of open subsets that cover $X\times X$ 
on each of which there is a section 
to the evaluation map $e\colon X^I \to X\times X$ given by $\alpha\mapsto (\alpha(0),\alpha(1))$. These local sections are called motion planners as they locally solve the motion
planning problem for the open neighborhoods on which they are defined. This concept 
can be extended by considering $n-2$ intermediate points $\frac{1}{n-1},\frac{2}{n-1},\ldots,\frac{n-2}{n-1}$ in the unit interval
$I$ and the multi-evaluation map $e\colon X^I \to X^n$ given by: 
\[ \alpha\mapsto \left(\alpha(0),\alpha\left(\frac{1}{n-1}\right),\ldots,\alpha\left(\frac{n-2}{n-1}\right),\alpha(1)\right). \] 
The resulting number is a homotopy invariant and is denoted by $\mathrm{TC}_n(X)$. 
Notice that a motion planner in this latter context finds a 
path in $X$ visiting a sequence of points in $X$ following a prescribed order, 
let us say $a_1$ first, $a_2$ next, and so on.    
Our aim is to consider only ``symmetric'' planners, or more precisely bidirectional planners. 
That is, planners such that the path they assign to 
$a_n,\dots,a_1$ (visited in this order) is the reverse of the 
path assigned to $a_1,\ldots,a_n$. 
This latter requirement may not be satisfied by the planners that define $\mathrm{TC}_n$; 
one simple reason is because their open sets need not be ``symmetric'' --- 
$(a_1,\ldots,a_n)$ and $(a_n,\ldots,a_1)$ may be in different domains. 
One way to alleviate this 
problem is to impose both the open sets and the planners be ``symmetric'' with respect to
the involution that reverses both the order of the components of an $n$-tuple in $X^n$ and the direction of a path in $X^I$. This leads us to define $\mathrm{TC}^\beta_n(X)$, the
bidirectional topological complexity of $X$ which we will show is a homotopy 
invariant of $X$. The ``symmetric'' condition for motion planners is also considered 
in a more elaborate way in the definition of ``symmetrized'' topological 
complexity $\mathrm{TC}^\Sigma_n(X)$ given in~\cite{bgrt}. These notions of 
topological complexity satisfy the inequalities
\[ \mathrm{TC}_n(X) \leq \mathrm{TC}^\beta_n(X) \leq \mathrm{TC}^\Sigma_n(X), \]
and moreover, $\mathrm{TC}^\beta_2(X) = \mathrm{TC}^\Sigma_2(X)$. 
The definition of $\mathrm{TC}^\beta_n$ should be thought of as the first natural simplification 
of $\mathrm{TC}^\Sigma_n$; or more precisely, as the first term of a sequence of ``symmetrized'' 
topological complexities interpolating between $\mathrm{TC}_n$ and $\mathrm{TC}^\Sigma_n$ (see Remark~\ref{last}).
The definition of bidirectional topological complexity is more relevant to the sequential motion
planning problem, it is also potentially easier to calculate, and could be used 
to understand and calculate some values of $\mathrm{TC}^\Sigma_n$. For instance, we prove the following:

\begin{theorem}\label{main_theorem}
\[ \mathrm{TC}^\beta_n(S^m) = \left\{
\begin{array}{cl}
n+1 & \text{ if } m \text{ is even},\\
n & \text{ if both } m,n \text{ are odd.}
\end{array}\right. \]
\end{theorem}

The calculation of $\mathrm{TC}^\beta_n(S^m)$ in Theorem~\ref{main_theorem} should be 
compared with the value of $\mathrm{TC}_n(S^m)$, 
which is equal to $n$ if $m$ is odd and equal to $n+1$ if $m$ is even (see \cite{bgrt}). 
Since $\mathrm{TC}^\beta_2(S^m)=\mathrm{TC}^\Sigma_2(S^m) = 3$ for all $m\geq 1$ (see \cite{g}), this implies that in general 
$\mathrm{TC}_n \neq \mathrm{TC}^\beta_n$. 
On the other hand, in all of our examples in this paper we have 
$\mathrm{TC}^\beta_n=\mathrm{TC}^\Sigma_n$, but we do not 
expect this to be generic. 
The distinction between these two could be addressed by means 
of equivariant obstruction theory, something that we plan to investigate elsewhere. 
Note that situations where $\mathrm{TC}^\beta_n=\mathrm{TC}^\Sigma_n$ 
indicate that there are scenarios for which the intuitively more 
difficult problem of finding fully ``symmetrized'' motion planners is equally difficult as the simpler bidirectional motion planning problem (see for
instance Example~\ref{bidirectional_symmetrized_same}).

\begin{remark}\normalfont
According to \cite{g}, $\mathrm{TC}^\Sigma_2(S^m) = 3$ for all $m$, and $\mathrm{TC}^\Sigma_n(S^m) = n+1$ when $m$ is even. Should $\mathrm{TC}^\beta_n(S^m) = n$ when $n$ is even and $m$ is odd, then we would have an
example for which $\mathrm{TC}^\beta_n \neq \mathrm{TC}^\Sigma_n$. Moreover, if this latter was the case when $m=1$,
then this would allow us to calculate the symmetrized topological complexity of the 
2-dimensional torus since $4\geq \mathrm{TC}^\beta_4(S^1) \geq \mathrm{TC}^\Sigma_2(T^2) \geq 4$
(see Proposition~\ref{bidirectional_bounded_below_by_symmetric} and \cite{jg}).
\end{remark}

We will show that bidirectional topological complexity can be estimated by considering 
a suitable symmetric product, in a similar way as $\mathrm{TC}^\Sigma_2(X)$ can be estimated by 
considering $SP^2(X)$ (see~\cite{jg} and~\cite{g}). Extending these ideas we 
are able to provide lower cohomological bounds for both $\mathrm{TC}^\beta_n$ and $\mathrm{TC}^\Sigma_n$
when $n$ is even. We get the following calculations:

\begin{theorem}\label{second_theorem} 
If $n$ is even and $m>1$, then $\mathrm{TC}^\Sigma_n(S^m) = n+1$.
\end{theorem}

\begin{theorem}\label{third_theorem}
If $n$ is even and $m = 2^e$ with $e>1$, then 
$\mathrm{TC}^\Sigma_n(\mathbb{R}\mathrm{P}^m) = nm+1$. 
\end{theorem}

The organization of this paper is as follows: in Section 2 we will develop the 
necessary concepts to define two versions of bidirectional topological complexity
and obtain properties similar to those of the symmetrized and the symmetric 
topological complexity. Then in Section 3 we will obtain cohomological lower bounds for
$\mathrm{TC}^\beta_n$ and $\mathrm{TC}^\Sigma_n$ when $n$ is even; 
and in Section 4 we will describe specific motion planners that realize
the calculation $\mathrm{TC}^\Sigma_2(S^m)=3$ obtained in~\cite{g}. Throughout this paper:
(1) we will assume that all topological spaces under consideration are connected, and (2) 
we will use the unreduced definition of topological complexity and sectional category.

%%%%%%%%%%%%%%%%%%%%%%%%%%%%%%%%%%%
%%%%%%%%%%%%%%%%%%%%%%%%%%%%%%%%%%%

\section{Bidirectional Topological Complexity} 

\subsection{Motivation and Definition}
There are two versions of ``symmetric'' topological complexity: 
the symmetric and the symmetrized. 
The symmetric defined in~\cite{fg} is one greater than 
the sectional category of the quotient fibration of the pullback 
of the fibration $e:X^I\to X\times X$ induced by the inclusion of the configuration space $F(X,2)=\{ (x,y)\in X\times X: x\neq y\}$ into $X\times X$. That is $\mathrm{TC}^S_2$, is one more the sectional category of 
the map $\epsilon_2$ in the following diagram
\[
\xymatrix{
X^I\ar[d]^e & P_{F(X,2)}\ar[d]\ar[r]\ar[l] & P_{F(X,2)}/\mathbb{Z}_2 \ar[d]^{\epsilon_2}\\
X\times X & F(X,2)\ar[r]\ar[l] & B(X,2) }
\]
where $P_{F(X,2)}$ is the total space of the pullback and $B(X,2)$ is the
2-braid space of $X$. 
The resulting number is denoted by $\mathrm{TC}^S_2(X)$. This definition intends to provide efficient planners, in the sense that: (1) if a planner is to connect $a\in X$ with itself, then it will do so by means of a constant path; and (2) if a planner allows us to go from $a_1$ to $a_2$ in $X$, then this latter planner will use the same path, but in reverse direction, to go from $a_2$ to $a_1$. Unlike the classical topological complexity $\mathrm{TC}_2(X)$, 
the number $\mathrm{TC}^S_2(X)$ is not a homotopy invariant (see~\cite{fg}). 
We refer the reader to \cite{bgrt} and \cite{r} for two possible definitions of the symmetric topological 
complexity $\mathrm{TC}^S_n$ for higher values of $n$.  

The second version of ``symmetric'' topological complexity is called the symmetrized topological complexity of $X$, denoted by $\mathrm{TC}^\Sigma_n(X)$, which is defined in \cite{bgrt} 
as the equivariant sectional category of the multievaluation map $e\colon X^{J_n}\to X^n$, 
where $J_n$ is the wedge of $n$ copies of the unit interval $I$ and the symmetric group $\Sigma_n$ is acting on these spaces by permutation. 
This definition turns out to be a homotopy invariant, but unfortunately when $n>2$ 
it is not clear how this definition is related to the motion planning problem which is
supposed to be the source of inspiration for all of these types of invariants.

We propose to remedy this by defining a new homotopy invariant closely related 
to $\mathrm{TC}^S_2$ and therefore to the motion planning problem. To do this we start
by identifying the generator of $\mathbb{Z}_2 = \langle \beta \rangle$ with the
permutation of $\Sigma_n$ given by $\beta(i) = n + 1 -i$. In terms of transpositions this permutation 
is given by
\[  \beta = \left\{ \begin{array}{ll}
(1\ n)\cdots (k\ k+1) & \mbox{if } n=2k \\
(1\ n)\cdots (k\ k+2)(k+1) & \mbox{if } n=2k+1
\end{array} \right. .\]

Then we will let $\mathbb{Z}_2$ act on $X^n$ by 
$\beta\cdot (x_1,\dots,x_n) = (x_{\beta(1)},\dots,x_{\beta(n)})
= (x_n,...,x_1)$, and on the space $X^I$ by $(\beta\cdot \alpha)(t) = \alpha(1-t)$. 
This way the multievaluation map becomes a $\beta$-equivariant map.

We will use the ideas mentioned above to define what perhaps should have been called symmetric topological complexity. To avoid clashing with the nomenclature already 
chosen in \cite{bgrt} and~\cite{fg}, we will use the word ``bidirectional.''  
An open set $U\subset X^n$ will be called $\beta$-symmetric if $\beta(U) = U$, 
and if $s$ is a $\beta$-equivariant local section of the multievaluation map then
we will say that $s$ is a bidirectional motion planner.

\begin{definition}\normalfont The $n$-th bidirectional topological complexity of a space $X$, denoted by $\mathrm{TC}^\beta_n(X)$, is the smallest number of $\beta$-invariant open subspaces that
cover $X^n$ on each of which there is a bidirectional motion planner.
\end{definition}  

In other words, the bidirectional topological complexity is the $\beta$-equivariant
sectional category of the multievaluation fibration $e\colon X^I \to X^n$. Note that this is now 
directly related to the motion planning problem, since a local section of this map 
will be a local planner over a $\beta$-symmetric neighborhood realizing paths that can be run in either direction.

%%%%%%%%%%%%%%%%%%%%%%%%%%%%%%%%%%%%%%%%%%%%%%%%%%%%%%%%%%%%%%%%%%%%%%%%%%%%%%%%%%%
%%%%%%%%%%%%%%%%%%%%%%%%%%%%%%%%%%%%%%%%%%%%%%%%%%%%%%%%%%%%%%%%%%%%%%%%%%%%%%%%%%%
\subsection{Properties}

It is not hard to see that the definition of bidirectional topological complexity 
$\mathrm{TC}^\beta_2$ agrees with the 
symmetrized topological complexity $\mathrm{TC}^\Sigma_n$ when $n=2$.
Moreover, according to \cite{bgrt} we know that 
\[ \mathrm{TC}^S_2(X)-1\leq \mathrm{TC}^\Sigma_2(X) = \mathrm{TC}^\beta_2(X) \leq \mathrm{TC}^S_2(X). \]

In this regard, the definition of $\mathrm{TC}^S_2$ can be generalized 
to the bidirectional setting as follows.
Consider the diagram 
\[
\xymatrix{
X^I\ar[d]^e & P_{F(X,n)}\ar[d]\ar[r]\ar[l] & P_{F(X,n)}/\beta \ar[d]^{\epsilon_n}\\
X^n & F(X,n)\ar[r]\ar[l] & F(X,n)/\beta }
\]
where $P_{F(X,n)}\to F(X,n)$ is the pullback of the first 
vertical map and $\epsilon_n$ is the resulting map on the quotients. 
Then define $\mathrm{TC}^b_n(X) = 1 + \mathrm{secat}(\epsilon_n)$. 
The above inequalities can be generalized to:

\begin{proposition}\label{2-symmetric} When $X$ is a CW-complex we have
\[ \mathrm{TC}^b_n(X)-1\leq \mathrm{secat} (X\stackrel{\Delta}{\to} X^n/\beta)
\leq \mathrm{TC}^\beta_n(X) \leq \mathrm{TC}^b_n(X). \]
\end{proposition}
\begin{proof}
The last inequality can be obtained by following verbatim
Corollary 9 in~\cite{fg}. Now the diagonal 
inclusion of $X$ into $X^n/\beta$ can be replaced using the following commutative diagram
\[ \xymatrix{
X\ar[r]^c \ar[rd]_\Delta & X^I/\beta \ar[d] \\
 & X^n/\beta }\] 
and noticing that the map $c:X\to X^I/\beta$, that sends a point $x\in X$ to the
class of the constant map $c_x$, is a homotopy equivalence.
We also have a commutative diagram
\[ \xymatrix{
X^I\ar[d]^e \ar[r] & X^I/\beta\ar[d] & P_{F(X,n)}/\beta \ar[l] 
\ar[d]^{\epsilon_n}\\
X^n \ar[r] & X^n/\beta & F(X,n)/\beta \ar[l]}\]
The second inequality is easily obtained by noticing that
a bidirectional motion planner induces a local section for the middle map
and hence for $X\stackrel{\Delta}{\to} X^n/\beta$. The first one follows
from the fact that the second square in the above diagram is a pullback.
\end{proof}

\begin{remark} \normalfont
We will not further develop the topological complexity $\mathrm{TC}^b_n(X)$ as 
this is likely
not going to be a homotopy invariant of $X$ when $n>2$. 
However, note that when $n=2$ the space $X^2/\beta$ is the symmetric square
$SP^2(X)$, and so we obtain the following: 
\[ \mathrm{TC}^S_2(X)-1\leq \mathrm{secat} (X\stackrel{\Delta}{\to}SP^2(X))\leq \mathrm{TC}^\Sigma_2(X) \leq \mathrm{TC}^S_2(X). \]
The second inequality from the left hand side had been already noticed in ~\cite{jg}.
These inequalities show that $\mathrm{secat}(X\stackrel{\Delta}{\to} SP^2(X))$ is within one unit of 
$\mathrm{TC}^\Sigma_2(X)$. Moreover, note that the sectional category of 
$X\stackrel{\Delta}{\to} SP^n(X)$ bears its own significance as it is also a homotopy invariant 
of $X$. We will see later that this invariant is also related to bidirectional topological complexity.
\end{remark}

A natural question to ask at this point is whether the bidirectional complexity $\mathrm{TC}^\beta$ relates 
to the symmetrized complexity $\mathrm{TC}^\Sigma$. In this regard, we have the following result.
 
\begin{proposition} $\mathrm{TC}^\beta_n(X) \leq \mathrm{TC}^\Sigma_n(X)$, for all $n\geq 2$. 
\end{proposition}
\begin{proof}
If $\alpha$ is in $X^{J_n}$ then it defines an $n$-tuple of paths $(\alpha_1,\ldots,\alpha_n)$ with $\alpha_i(0)=\alpha_j(0)$ for all $i,j$; and if 
$\gamma$ is in $X^I$ we  can think of it as a sequence of concatenated paths 
$\gamma_1,\gamma_2,\ldots,\gamma_{n-1}$ determined by the distinguished points 
$\{ \frac{i}{n-1} \in I: i=0,1,\ldots,n-1 \}$.
Let $f:X^{J_n} \to X^I$ be given by $\alpha \mapsto \gamma_1\cdots\gamma_{n-1}$
where $\gamma_i=\alpha_i^{-1}\cdot \alpha_{i+i}$. 
Note that $f$ is $\beta$-equivariant and commutes with
the multievaluation maps. So we have a commutative diagram
\[ \xymatrix{
X^{J_n} \ar[rr]^f \ar[rd]_e & & X^I \ar[ld]^e \\
 & X^n & }\]
Therefore a $\Sigma_n$-equivariant section over a $\Sigma_n$-symmetric subset of 
$X^n$ will yield a $\beta$-equivariant section over a $\beta$-equivariant subset of 
$X^n$. The result follows.
\end{proof}

Before obtaining more properties about bidirectional topological complexity we need
to recall some facts about equivariant category and its relationship to 
equivariant sectional category. For simplicity, we will assume throughout this paper 
that all groups are finite and that any $G$-space is compact. We refer the reader to \cite{g} and \cite{gh} for
more details.

If $B$ is a $G$-space, then an open set $U\subset B$ is called $G$-categorical if 
the inclusion map $U\to B$ is $G$-homotopic to a map with values in a 
single orbit. The $G$-category of $B$ is denoted by $\mathrm{cat}_G(B)$ and is 
the smallest number of $G$-categorical open sets that cover $B$. Likewise, 
the equivariant sectional category of a $G$–map $p\colon E \to B$, denoted
$\mathrm{secat}_G(p)$, is the least integer $k$ such that $B$ may be covered by $k$ $G$-invariant 
open sets $U\subset B$ on each of which there exists a $G$–map $s\colon U \to E$ such that
$p\circ s$ is $G$-homotopic to the inclusion map $i\colon U\to B$. 
If $p$ is a $G$-fibration, then this latter condition can be replaced by $ps = i_U$.

\begin{proposition}\label{secat_less_cat}
\cite[Proposition 4.5]{gh}
Let $p \colon E \to B$ be a $G$–map. If $p(E^H)=B^H$ for all closed subgroups $H$ of $G$, then $\mathrm{secat}_G(p)\leq cat_G(B)$.
\end{proposition}

In order to get stronger upper bounds for $\mathrm{TC}^\beta$ we will make use of 
the following result.

\begin{theorem}\label{conn}
\cite[Theorem 3.5]{g}
Let $p\colon E \to B$ be a Serre $G$-fibration with fibre $F$, whose base
$B$ is a $G$-CW complex of dimension at least 2. Assume that $\pi_j(F^H) = 0$ 
for all subgroups $H \leq G$ and all $j < s$, where $s\geq 0$. 
Then
\[ \mathrm{secat}_G(p) < \frac{\dim B + 1}{s + 1} +1 .\]
\end{theorem}

We are now in position to derive more properties about bidirectional 
topological complexity. They are summarized in the following result.

\begin{proposition}\label{basic} Suppose $X$ and $Y$ are two topological spaces. We have:
\begin{itemize}
\item[a)] If $X\simeq Y$, then $\mathrm{TC}^\beta_n(X) = \mathrm{TC}^\beta_n(Y)$.
\item[b)] $\mathrm{TC}^\beta_{2n}(X)\leq \mathrm{TC}^\beta_{2n+1}(X)$.
\end{itemize}
When $X$ is a finite CW-complex, we have:
\begin{itemize}
\item[c)] $\mathrm{TC}_n(X) \leq \mathrm{TC}^\beta_n(X) \leq \mathrm{TC}^\Sigma_n(X)\leq n\dim(X)+1$.
\item[d)] If $X$ is $q$-connected, then 
\[ \mathrm{TC}^\beta_n(X) < \frac{n \dim X +1}{q+1} + 1.\]
\end{itemize}
\end{proposition}
\begin{proof}
The first statement can be proved following the corresponding 
arguments in \cite[Proposition 4.7]{bgrt}. 
For the second statement pick a point $x_0$ in $X$. Now it suffices to note that 
a bidirectional motion planner for $e_{2n+1}:X^I\to X^{2n+1}$ precomposed with the 
$\beta$-equivariant map $\phi(x_1,\ldots,x_{2n})=(x_1,\ldots,x_n,x_0,x_{n+1},\dots,x_{2n})$
yields a bidirectional motion planner for $e_{2n}:X^I\to X^{2n}$.

For (c), we want to apply
Proposition \ref{secat_less_cat} to the $\Sigma_n$-map $e\colon X^{J_n}\to X^n$.
For this, note that if $H$ is a subgroup of $\Sigma_n$, then a multipath 
$\alpha=(\alpha_1,\ldots,\alpha_n)$ is in $(X^{J_n})^H$ when its components
$\alpha_i$ are identical for all indexes $i$ that are moved by permutations 
in $H$. Likewise, $(x_1,\ldots,x_n)$ is in $(X^n)^H$ if and only if its
components $x_i$ are identical for all indexes that are moved by permutations
in $H$. As $X^n$ is path connected, we have $e((X^{J_n})^H) = (X^n)^H$ for 
all subgroups $H$ of $\Sigma_n$. Then, $\mathrm{TC}^\Sigma_n(X)\leq cat_{\Sigma_n}(X^n)$. 
Now, according to Corollary 1.12 in~\cite{m} (see also~\cite{colman}), since $X^n$ is a finite complex 
and $(X^n)^{\Sigma_n} \cong X$ is connected, it follows that
$cat_{\Sigma^n}(X^n) \leq \dim(X^n/\Sigma_n) +1$. Part (c) follows since
$\dim(X^n/\Sigma_n)\leq n\dim(X)$.
     
To prove (d), note that the multievaluation map 
$e \colon X^I \to X^n$ is induced by the inclusion map 
$\{ \frac{i}{n-1} \in I: i=0,1,\ldots,n-1 \} \hookrightarrow I$. Moreover, this latter is a 
$\beta$-cofibration and hence the map $e$ is a $\beta$-fibration 
(see \cite[Example 2.4 and Proposition 2.5]{g}). Recall
also that the fiber of $e$ is equivalent to $(\Omega X)^{n-1}$, and note 
that the action of $\beta$ on the fiber
is given by $\beta\cdot (\alpha_1,\ldots,\alpha_{n-1})=
(\alpha_{n-1}^{-1},\ldots,\alpha_1^{-1})$. This implies that the fixed-point set
$[(\Omega X)^{n-1}]^\beta$ can be identified with 
$(\Omega X)^{k} \times P_0 X$ when $n-1=2k+1$ and with 
$(\Omega X)^{k}$ when $n-1=2k$, where $P_0 X$ is the space of paths
in $X$ that start at a given point in $X$. Therefore the fixed-point set is 
$(q-1)$-connected, and we can apply Theorem~\ref{conn} to obtain (d). \end{proof}

\begin{remark}\normalfont
Parts (c) and (d) in the previous Proposition can also be derived from Theorem 7.1 in \cite{g}.
\end{remark}

\begin{example}\normalfont 
Suppose $g\geq 2$ and $m\geq 2$, and let $g\mathbb{R}\mathrm{P}^m$ be the connected sum
of $\mathbb{R}\mathrm{P}^m$ with itself $g$ times. 
Then by Proposition~\ref{basic}(c),Theorem 1.3 of~\cite{cv} and Theorem 1.1 of~\cite{ag}, 
the value of $\mathrm{TC}_n, \mathrm{TC}^\beta_n$ and $\mathrm{TC}^\Sigma_n$ of 
$g\mathbb{R}\mathrm{P}^m$ equals $nm+1$ for all $n\geq 2$.

Similarly, using the calculations in ~\cite{ggghmr}, the value of $\mathrm{TC}_n, \mathrm{TC}^\beta_n$ and 
$\mathrm{TC}^\Sigma_n$ of an orientable, closed, connected surface of genus $g$ equals
$2n+1$ for all $n,g\geq 2$. 
\end{example}

\begin{example}\label{even_spheres} \normalfont
A test calculation is always that of the topological complexity of spheres. As an application of the previous result we have
\[ \mathrm{TC}^\beta_n(S^m) < \frac{nm + 1}{m} + 1. \]
It is well-known that $\mathrm{TC}_n(S^m)=n+\delta$, where $\delta$ is equal to $1$ when
$m$ is even, and equal to $0$ when $m$ is odd (see \cite{r}).
Then it follows that
$n\leq \mathrm{TC}^\beta_n(S^m) \leq n+1 $.
This latter uncertainty can be settled when $m$ is even yielding
\[ \mathrm{TC}^\beta_n(S^m) = n+1. \]
\end{example}

Before calculating other values of $\mathrm{TC}^\beta_n(S^m)$ let us recall that it is possible to
get upper bounds for $\mathrm{TC}_n(X)$ when $X$ is a polyhedron by constructing local sections of 
the multievaluation map over a cover of $X^n$ consisting of Euclidean Neighborhood Retracts 
(ENRs) (see Proposition 2.2 in~\cite{r}). In the equivariant setting we have the following 
definition.

\begin{definition}
A $G$-space $X$ is a $G$-ENR if $X$ is $G$-homeomorphic to a $G$-retract of some open 
$G$-subspace in a orthogonal representation of $G$. 
\end{definition}

The following result, which is the equivariant version of
Corollary 8.7 in Chapter 4 of~\cite{d}, will allow us to obtain an upper bound for 
$\mathrm{TC}^\beta_n(S^m)$.

\begin{lemma}\label{g_enr}
Suppose that both $X$ and $Y$ are $G$-ENRs and that $X$ is a $G$-subspace of $Y$. Then there is 
a $G$-invariant open subspace $U$ of $Y$ containing $X$, a $G$-retraction $r\colon U\to X$ and a
$G$-homotopy $\Theta$ between the inclusion map $i_U\colon U\to Y$ and $i_X\circ r\colon U \to Y$,
where $i_X$ is the inclusion map of $X$ in $Y$.  
\end{lemma}
\begin{proof}
Suppose that $Y$ is equivariantly embedded in a orthogonal $G$-representation $\mathbb{R}^N$.  
As noted in~\cite{bgrt}, since $X$ is $G$–equivariantly embedded in $\mathbb{R}^N$, 
there exists a $G$–invariant open neighborhood $V$ of $X$ in $\mathbb{R}^N$ and a 
$G$–equivariant retraction $\rho \colon V\to X$ (see also Proposition 5.2.1 of~\cite{td}). 
Following the proof of Proposition 8.6 in Chapter 4 of~\cite{r}, let $W$ be the subset of $V$ 
consisting of points $y\in V$ such that the line segment from $y$ to $\rho(y)$ lies in $V$, and
let $U = Y\cap W$. 
Then $U$ is $G$-invariant open subspace of $\mathbb{R}^N$ containing $X$, and $r\colon U \to X$ given 
by $\rho$ restricted to $U$ is a $G$-retraction. Moreover $\Theta(y,t)=(1-t)y +tr(y)$ is
a $G$-homotopy between $i_U$ and $i_X\circ r$ as wanted.
\end{proof}

\noindent{\bf Proof of Theorem~\ref{main_theorem}}

The case when $m$ is even was treated in Example~\ref{even_spheres}. We will prove the case when both 
$m$ and $n$ are odd, but first we need to set some notation. For each $d=0,\ldots,n-1$
we define $F_{n,d}$ as the subspace of $(S^m)^n$ consisting of $n$-tuples $(x_1,\ldots,x_n)$ 
for which there is a set
of exactly $d$ indices $1\leq i_1< \ldots < i_d\leq n-1$ such that $x_{i_j} = x_{i_j+1}$ 
for each $i_j$. 
For instance, $F_{n,0}$ and $F_{n,n-1}$ are the configuration space and diagonal of $S^m$ respectively. 

Now asssume that both $m$ and $n$ are odd. For each $0\leq j \leq n-1$, we let
\[V_j= \{ (x_1,\ldots,x_n) \in (S^m)^n : (x_1,-x_2,x_3,-x_4,\ldots,-x_{n-1},x_n)\in F_{n,j} \} .\]
Note that the subspaces $V_j$ are $\beta$-invariant and provide a cover of $S^m$. 
Pick a non-vanishing vector field $v$ on $S^m$, write $n=2l+1$, 
and for each $n$-tuple $(x_1,\ldots,x_n)$ define a $\beta$-motion planner as follows:
\begin{enumerate}
\item if $x_{i} = x_{i+1}$, then use a constant path;
\item if $x_{i} \neq - x_{i+1}$, then use the shortest path on $S^m$ that joins these two points;
\item if $x_{i} = - x_{i+1}$, then 
\begin{enumerate}
\item when $i\leq l$: use the vector field $v$ to travel from $x_{i}$ to $x_{i+1}$ through the
great arc in the direction of $v(x_i)$;
\item and when $i\geq l+1$: travel backwards through the great arc from $x_{i+1}$ to $x_i$ in the 
direction of $v(x_{i+1})$.
\end{enumerate}
\end{enumerate}
These rules define bidirectional motion planners on each of the subspaces $V_j$.
Now note that the fixed point sets $((S^m)^n)^{\langle 1 \rangle}$, 
$((S^m)^n)^{\langle \beta \rangle}$, $V_j^{\langle 1 \rangle}$, and $V_j^{\langle \beta \rangle}$ 
are ENRs since they are locally compact and locally contractible. 
Hence by Theorem 4.10 of~\cite{bgrt} it follows that they are $\beta$-ENRs.
Then, by applying Lemma~\ref{g_enr} to each $V_j$ we can extend these rules to define $n$ bidirectional motion on open $\beta$-invariant subsets of $X^n$, which in turn implies that 
$\mathrm{TC}^\beta_n(S^m)\leq n$. 
The result follows from the inequality $n=TC(S^m)\leq \mathrm{TC}^\beta_n(S^m)$ when $m$ is odd.
\hfill $\square$

%%%%%%%%%%%%%%%%%%%%%%%%%%%%%%%%%%%%%%%%%%%%%%%%%%%%%%%%%%%%%%%%%%%%%%%%%%%%%%%%%%%
%%%%%%%%%%%%%%%%%%%%%%%%%%%%%%%%%%%%%%%%%%%%%%%%%%%%%%%%%%%%%%%%%%%%%%%%%%%%%%%%%%%

\subsection{Symmetric Products and Topological Complexity}

As noted in~\cite{g}, the symmetrized topological complexity $\mathrm{TC}^\Sigma_2(X)$ 
is intimately related to the symmetric product $SP^2(X)$. The following result
extends this to higher topological complexities. To simplify notation we will 
write $\beta P^n (X)$ 
instead of $X^n/\beta$ which we will call the $n$-th bidirectional product of $X$, 
and note that its homotopy type depends only on that of $X$. 

\begin{proposition}\label{lower_bounds}
If $X$ a CW-complex, then: 
\begin{enumerate}
\item $\mathrm{TC}^\Sigma_n(X) \geq \mathrm{secat}(X\to SP^n(X))$, 
\item $\mathrm{TC}^\beta_n (X) \geq \mathrm{secat} (X\to \beta P^n(X))$, and 
\item $\mathrm{secat}(\beta P^n(X)\to  SP^n(X)) \leq \mathrm{secat}(X\to SP^n(X))$. 
\end{enumerate}
\end{proposition}
\begin{proof}
The first two inequalities can be proved as in Proposition~\ref{2-symmetric}.
The last one follows from the commutative and equivariant diagram
\[ \xymatrix{
 & X\ar[rd]^\Delta\ar[ld]_\Delta & \\
\beta P^n(X) \ar[rr] & & SP^n(X)
}\]
where $\Delta$ stands for the corresponding diagonal map and the bottom map is the 
natural map $X^n/\beta \to X^n/\Sigma_n$ 
(see Proposition 2.1 in \cite{r}).
\end{proof}

\begin{example}\label{bidirectional_symmetrized_same}\normalfont
According to \cite[Theorem 5.7]{ggl}, $\mathrm{TC}_n(\mathbb{R}\mathrm{P}^m)=nm+1$ when $m$ is 
even and $n>m$. Thus, by Proposition \ref{basic}, we have
\[ nm+1=\mathrm{TC}_n(\mathbb{R}\mathrm{P}^m) \leq TC^\beta_n(\mathbb{R}\mathrm{P}^m)\leq 
\mathrm{TC}^\Sigma_n(\mathbb{R}\mathrm{P}^m)\leq n\dim(\mathbb{R}\mathrm{P}^m)+1= nm+1.\]
That is, 
$\mathrm{TC}_n(\mathbb{R}\mathrm{P}^m) = TC^\beta_n(\mathbb{R}\mathrm{P}^m) = 
\mathrm{TC}^\Sigma_n(\mathbb{R}\mathrm{P}^m) = nm+1$ when $m$ is even and $n>m$.

Moreover, it is known that the inclusion
$\mathbb{R}\mathrm{P}^{2}\stackrel{\Delta}{\to} SP^2(\mathbb{R}\mathrm{P}^{2})\simeq \mathbb{R}\mathrm{P}^4$ 
is null-homotopic (see the comment before Example 4.1 in \cite{g}). Then Proposition~\ref{lower_bounds} yields 
$\mathrm{TC}_2^\Sigma(\mathbb{R}\mathrm{P}^2)\geq \mathrm{secat}(\mathbb{R}\mathrm{P}^2 \stackrel{\Delta}{\to} SP^2(\mathbb{R}\mathrm{P}^2))$,
and this latter in turn is greater than the cup-length of $\mathbb{R}\mathrm{P}^4$.
Thus $\mathrm{TC}_2^\Sigma(\mathbb{R}\mathrm{P}^2) = 5$. Summarizing, we have
\[ \mathrm{TC}^\Sigma_{n}(\mathbb{R}\mathrm{P}^{2}) = 2n+1, \]
for all $n\geq 2$. 
Note that this is example shows that there are spaces for which the motion planning problem in the symmetrized setting, in the bidirectional setting, and in the ordinary case are
equally difficult to solve. 
\end{example}

Let us record a couple of properties of bidirectional products in the following
lemmas.

\begin{lemma}\label{factor}
When $n > 1$ is odd the space $\beta P^n(X)$ is homeomorphic to $\beta P^{n-1}(X)\times X$. 
\end{lemma}
\begin{proof}
We can easily prove this result by just noticing that the middle copy of $X$
in $X^n$ is fixed by the action of $\beta$.
\end{proof}

\begin{lemma}\label{sym_prod}
The space $SP^2(X^l)$ is homeomorphic to $\beta P^{2l}(X)$. 
\end{lemma}
\begin{proof}
Consider the map $\varphi:(X^{l})^2 \to X^{2l}$ given by
$(x_1,\ldots,x_{2l}) \mapsto (x_l,\ldots,x_1,x_{l+1},\ldots,x_{2l})$.
This map is its own inverse and respects the corresponding actions. Therefore
it defines an equivariant homeomorphism that passes to the quotients  
$SP^2(X^l)$ and $\beta P^{2l}(X)$.
\end{proof}

The following result could potentially be useful in practical scenarios since 
it allows us to bypass the symmetrized setting as we will see in the example right after.

\begin{proposition}\label{bidirectional_bounded_below_by_symmetric}
We have $\mathrm{TC}^\Sigma_2(X^l) \leq \mathrm{TC}^\beta_{2l}(X)$.
\end{proposition}
\begin{proof}
For simplicity of notation we will only consider the case $l=2$. We have
 the following commutative and equivariant diagram
\[ \xymatrix{
X^I \ar[r] \ar[d]^{e_{0,1/3,2/3,1}}  & (X\times X)^I \ar[d]^{e_{0,1}}  \\
X\times X\times X\times X \ar[r]_\varphi & (X\times X)\times (X\times X) } \]
The top map in this diagram sends a path $\alpha$ to $(\alpha_1,\alpha_2)$ 
where $\alpha_1$ is the reparametrization to [0,1] of $\alpha |_{[1/3,2/3]}$ 
and $\alpha_2$ is $\alpha$ again, and $\varphi$ is the equivariant homeomorphism
of the previous Lemma. Therefore a local section of 
the vertical map on the left hand side will induce one for the vertical
map on the right hand side. The result follows. 
\end{proof}

\begin{example} \normalfont
Suppose we want to consider sequential bidirectional planners of order $2l$ on 
the sphere $S^m$, when $m$ is even. According to the previous result, the number 
of such planners satisfies $\mathrm{TC}^\beta_{2l}(S^m)\geq \mathrm{TC}^\Sigma_2((S^m)^l)$. 
When $m$ is even this latter lower bound is greater than or equal to $\mathrm{TC}_2((S^m)^l)=2l+1$ according
to Corollary 3.12 in~\cite{bgrt}, and  
by connectivity we also know that $\mathrm{TC}^\beta_{2l}(S^m))\leq 2l+1$. Thus, when
$m$ is even we see that
\[\mathrm{TC}^\Sigma_2((S^m)^l) = \mathrm{TC}^\beta_{2l}(S^{m}) = 2l+1. \]
This is suggesting that, at least for highly connected spaces, 
bidirectional planning is as difficult as symmetrized planning in 
cartesian products.   
\end{example}

%%%%%%%%%%%%%%%%%%%%%%%%%%%%%%%%%%%%%%%%%%%%%%%%%%%%%%%%%%%%%%%%%%%%%%%%%%%%%%%%%%%
%%%%%%%%%%%%%%%%%%%%%%%%%%%%%%%%%%%%%%%%%%%%%%%%%%%%%%%%%%%%%%%%%%%%%%%%%%%%%%%%%%%

\section{Cohomological Lower Bounds}

Recall that the cup-length of a space $X$, denoted $cl(H^*(X;\mathbb{F}))$, is the longest length of a nontrivial product in $\overline{H}^*(X;\mathbb{F})$. A class in 
$\overline{H}^*(X^n;\mathbb{F})$ 
is called a zero divisor if when restricted to $\overline{H}^*(X;\mathbb{F})$ 
by $\Delta^*$ we get 
the zero class. It is well known that
$\mathrm{TC}_n(X)$ is bounded below by the zero-divisors cup-length 
$zcl(H^*(X^n;\mathbb{F}))$. Likewise in the symmetrized case, 
since $\mathrm{TC}^\Sigma_n(X)$ is bounded below by the sectional category of the diagonal inclusion
$\Delta\colon X\to SP^n(X)$, it follows that $\mathrm{TC}^\Sigma_n(X)$ is bounded below by the 
``symmetrized'' zero-divisors cup-length: the cup-length of the kernel of 
$\Delta^*:H^*(SP^n(X);\mathbb{F})\to H^*(X;\mathbb{F})$. 
Notice that this is useful only when 
$n!$ is not invertible in the coefficients field $\mathbb{F}$, 
otherwise $H^*(SP^n(X);\mathbb{F}) \cong H^*(X^n;\mathbb{F})^{\Sigma_n} \subset 
H^*(X^n;\mathbb{F})$ which implies 
$zcl(H^*(SP^n(X);\mathbb{F})) \leq zcl(H^*(X^n;\mathbb{F}))$. 
This means that we need to consider the torsion part of the cohomology 
of $SP^n(X)$ if we expect stronger lower bounds. We will rely on the work of
Nakaoka for this end.

\begin{example}\normalfont
Using cohomological lower bounds one can check that 
$\mathrm{TC}_n(M^{2m})=nm+1$, where $M^{2m}$ is a simply connected 
symplectic manifold of dimension $2m$. By Proposition~\ref{basic}, it follows
that
\[ \mathrm{TC}_n(M)=\mathrm{TC}^\beta_n(M)=nm+1. \]
\end{example}

\begin{theorem}\label{lower_cohomology} 
If $X$ is a finite CW-complex, then 
\begin{enumerate}
\item $\mathrm{TC}^\beta_{2n}(X) \geq cl(H^*(SP^2(X^n);\mathbb{F}_2)) + 1$, and \\
\item $\mathrm{TC}^\Sigma_{2k}(X)\geq k\cdot\mathrm{cl}(H^*(SP^2(X;\mathbb{F}_2))+1$.
\end{enumerate}
\end{theorem}
\begin{proof}
By Lemma~\ref{factor} and Lemma~\ref{sym_prod} we can factor the
diagonal inclusion of $X\to \beta P^m(X)$ as follows:
\[ X\stackrel{\Delta}{\to} X^{n} \stackrel{\Delta}{\to} SP^2(X^{n})\cong 
\beta P^{2n}(X). \]
Since the map $Y\stackrel{\Delta}{\to} SP^2(Y)$ is trivial in reduced 
mod-2 cohomology for any finite CW-complex $Y$ according 
to~\cite[Theorem 11.2 and Theorem 11.4]{n}, the first inequality follows from 
Proposition~\ref{lower_bounds}.

Now let $H =\langle (1\ 2), (3\ 4),\ldots, (2k-1\ 2k)\rangle
\subset \Sigma_n$. Note that $H$ is a subgroup of $\Sigma_n$ isomorphic 
to $(\mathbb{Z}_2)^n$. Thus we have 
$\mathrm{secat}_{H}(e_n) \leq \mathrm{secat}_{\Sigma_n}(e_n) =\mathrm{TC}^\Sigma_n(X)$.
Moreover, we have the following commutative diagram 
\[ \xymatrix{
X^{J_n}\ar[r]\ar[d]_{e_n} & X^{J_n}/H \ar[d]_{\overline{e}_n} & X\ar[l]_\simeq \ar[ld]^\Delta \\
(X^2)^k  \ar[r] &(SP^2(X))^k   & 
}\]
where $\overline{e}_n$ is the quotient of the evaluation map $e_n$, and the
homotopy equivalence on the right hand side is induced by the inclusion of $X$ 
into $X^{J_n}$. 
So, an $H$-symmetrized motion planner will induce a section for $\overline{e}_n$.
Thus $\mathrm{secat}(\Delta)=\mathrm{secat}(\overline{e}_n)\leq \mathrm{secat}_H(e_n)$.
The diagonal map in this latter diagram factors as
\[ X \stackrel{\Delta}{\to} SP^2(X)  \stackrel{\Delta}{\to} (SP^2(X))^k\]
Thus, using the fact that $Y\stackrel{\Delta}{\to} SP^2(Y)$ is trivial in reduced
mod-2 cohomology, we obtain the second inequality.
\end{proof}

\begin{remark}\normalfont
The mod-2 cohomology of the bidirectional product $\beta P^{2n}(X)$ contains more
information than that of $SP^2(X)$ as the cohomology of this latter
injects into that of the bidirectional product as a direct summand through the 
projection map $\pi_{1,2n}$ induced by  
$(x_1,\ldots,x_{2n})\mapsto (x_1,x_{2n})$  
as can be seen in the following commutative diagram
\[ \xymatrix{
 & X \ar[d]^\Delta \ar[rd]^\Delta \ar[ld]_\Delta & \\
SP^2(X)\ar[r]^r & \beta P^{2n}(X) \ar[r]^{\pi_{1,2n}} & SP^2(X) 
}\]
where $r$ is induced by $(x_1,x_2)\mapsto (x_1,\ldots,x_1,x_2,\ldots,x_2)$, 
and the bottom composite satisfies $\pi_{1,2n}\circ r = 1$.
 
The above diagram allows us to see that 
$\mathrm{secat}(X\stackrel{\Delta}{\to} \beta P^{2n}(X))$ is bounded
below by the cup-length of the kernel of 
$H^*(\beta P^{2n}(X))\stackrel{r^*}{\to} H^*(SP^2(X))$.
Also note that we have more projections $\pi_{i,2n-i+1}$ which allow the
cohomology of the symmetric product be injected in different ways into that
of the bidirectional product.
\end{remark}

The following calculations are based 
on Nakaoka's analysis \cite{n}, which is distilled in 
~\cite{jg}, and allow us to estimate higher symmetrized topological 
complexities. We record here in a brief way what we need from \cite{jg} and
\cite{n}. 

For the rest of this section we will work with cohomology with
coefficients modulo-2. There are two homomorphisms: $E_s:H^*(X) \to H^{*+s}(SP^2(X))$ 
and $\phi:H^*(X \times X) \to H^*(SP^2(X))$ that satisfy the following:

\begin{theorem}\label{nakaoka_short}
\cite[Theorem 4.4]{jg} Let $\{ b_0, b_1,\dots ,b_m \}$ be a homogeneous basis for $H^*(X)$.
A basis for $H^*(SP^2(X))$ consists of 1, the elements $E_s(b_i)$ with 
$2 \leq s \leq deg(b_i)$, and
the elements $\phi(b_i \otimes b_j )$ with $i < j$. The ring structure is determined 
by the two relations:
\begin{enumerate}
\item $\phi(b_i \otimes b_j) \cdot \phi(b_u \otimes b_v) = 
\phi((b_i\cdot b_u) \otimes (b_j \cdot b_v)) + 
\phi((b_i \cdot b_v) \otimes (b_j \cdot b_u))$.
\item $E_s(b_i) \cdot \phi(b_u \otimes b_v) = E_s(b_i) \cdot E_t(b_j ) = 0$.
\end{enumerate}
\end{theorem}

The ring structure of $H^*(SP^2(X))$ is supplemented by more relations which are
listed in Theorem 4.4 of \cite{jg}. Of these we will need two of them:
\begin{itemize}
\item[(3)] $\phi(b_j \otimes b_i) = \phi(b_i \otimes b_j )$. 
\item[(4)] $\phi(b_i \otimes b_i) = \sum^{deg(b_i)}_{s=2} E_s(Sq^{deg(b_i)-s} b_i)$.
\end{itemize}

With this we are now ready to prove the last two Theorems presented in the Introduction.\\

\noindent{\bf Proof of Theorem~\ref{second_theorem}}

Let $e_m$ be the generator of $H^*(S^m)$. Thus 
$\phi(e_m\otimes 1)\phi(1\otimes e_m) = \phi(e_m\otimes e_m) = E_m(e_m)\neq 0$ in
$H^*(SP^2(S^m);\mathbb{F}_2))$ when $m>1$. 
Then by Theorem~\ref{lower_cohomology}
we have $n+1\leq \mathrm{TC}^\Sigma_n(S^m)$, and by the connectivity upper bound of 
Theorem 7.1 in \cite{g}, it follows that $\mathrm{TC}^\Sigma_n(S^m)\leq n+1$.
\hfill $\square$ \\

\noindent{\bf Proof of Theorem~\ref{third_theorem}}

The proof follows from Theorem~\ref{lower_cohomology} and the fact that the mod-2 cup-length of $SP^2(\mathbb{R}\mathrm{P}^m)$ is $2m$
when $m=2^e>1$, as was pointed out in \cite[Proposition 4.3]{g}.
\hfill $\square$ \\

\begin{remark}\normalfont 
The estimates obtained from Theorem~\ref{lower_cohomology} may not be strong enough
for $\mathrm{TC}^\beta_n$ as $n$ increases. For instance, the interested reader can check 
the following lower bounds for $\mathbb{R}\mathrm{P}^m$ with $m=2^e>1$ 
obtained from Theorem~\ref{nakaoka_short}:
\[ 4m \leq \mathrm{TC}^\beta_4(\mathbb{R}\mathrm{P}^m) \leq \mathrm{TC}^\Sigma_4 (\mathbb{R}\mathrm{P}^m) = 4m +1 \]
and
\[ 6m-2 \leq \mathrm{TC}^\beta_6(\mathbb{R}\mathrm{P}^m) \leq \mathrm{TC}^\Sigma_6 (\mathbb{R}\mathrm{P}^m) = 6m +1 \]
\end{remark}

\begin{remark}\label{last}\normalfont
The previous remark shows that $\mathrm{TC}^\beta$ may not be enough to describe
all the values of $\mathrm{TC}^\Sigma$ for a given space. At the same it suggests
that we could interpolate a chain of subgroups 
$1=G_0 \subset G_1 \subset \cdots \subset \Sigma_n$
and define the corresponding notions of $\mathrm{TC}^{G_i}_n$, where
$G_1=\langle \beta \rangle$. Note that $\mathrm{TC}^{G_0}_n = \mathrm{TC}_n$. 
This would yield a chain of inequalities
\[ \mathrm{TC}_n \leq \mathrm{TC}^{G_1}_n\leq \ldots\leq \mathrm{TC}^\Sigma_n. \]   
which should help capture the information between the ordinary and the
symmetrized topological complexity. Some of these intermediate complexities
may not be related to the motion planning problem at all. On the other hand,
each of these can be estimated from below by looking at the corresponding 
permutation product $G_iP^n(X):= X^n/G_i$. We will explore these ideas in
\cite{etg}. \end{remark}

%%%%%%%%%%%%%%%%%%%%%%%%%%%%%%%%%%%%%%%%%%%%%%%%%%%%%%%%%%%%%%%%%%%%%%%%%%%%%%%%%%%
%%%%%%%%%%%%%%%%%%%%%%%%%%%%%%%%%%%%%%%%%%%%%%%%%%%%%%%%%%%%%%%%%%%%%%%%%%%%%%%%%%%

\section{Planning on Spheres}

Calculating the values of the different versions of topological complexity becomes
potentially more applicable to practical problems when they are accompanied by explicit
motion planners. In this section we will describe bidirectional motion planners on the sphere 
$S^n\subset \mathbb{R}^{n+1}$ that realize the calculation of 
$\mathrm{TC}_2^\beta(S^n)=\mathrm{TC}^\Sigma_2(S^n)=3$ obtained in ~\cite{g}.
 
Let $\mathbf{n}=(0,\ldots,0,1)$, $\mathbf{s}=(0,\dots,0,-1)$ and $p_+,p_-:S^n \to \mathbb{R}^n$ be the stereographic projections with respect to $\mathbf{n}$ and $\mathbf{s}$ respectively. Let $D_+ = \{ x\in S^n|x_{n+1}>0\}$ and
$D_{-} = \{ x\in S^n|x_{n+1}<0\}$. 
Consider the following open sets of $S^n\times S^n$:
\[ U_+ = (S^n \setminus \{ \mathbf{n}\} )\times (S^n \setminus \{ \mathbf{n} \} ) \]
\[ U_- = (S^n \setminus \{ \mathbf{s}\} )\times (S^n \setminus \{ \mathbf{s} \} ) \]
\[ V = (D_+ \times D_-) \cup (D_-\times D_+) \]
Note that these three are open $\beta$-symmetric and cover $S^n\times S^n$. Moreover both $U_+$ and $U_-$ are contractible and hence there exists a bidirectional motion planners on each of them. For instance, on $U_+$ we can use the stereographic projection $p_+$ to create a bidirectional motion planner on $U_+$ (similarly on $U_-$ with $p_-$).
When $(x,y)$ is in $V$ we construct a path from $x$ to $y$ as follows: let $x\in S^n$ such that $x_{n+1}\neq 0$, and
consider
\[ \alpha_x(t) = \frac{(1-t)x + t(0,\ldots,0,\frac{x_{n+1}}{|x_{n+1}|})}{||(1-t)x + 
t(0,\ldots,0,\frac{x_{n+1}}{|x_{n+1}|}) ||} \]
and
\[ \omega_x(t) = (0,\ldots,0,\sin(\pi t),\frac{x_{n+1}}{|x_{n+1}|} \cos(\pi t)) \]

The path connecting $(x,y)\in V$ will be given by $H(x,y)(t) = [\alpha_x\cdot\omega_x\cdot\alpha_y](t)$.
This construction realizes the calculation of $\mathrm{TC}_2^\beta(S^n)=3$.


\begin{thebibliography}{CJS}

\bibitem{ag} 
Aguilar-Guzman, J. Gonzalez, J. Sequential motion planning in connected sums of real
projective spaces. arXiv:1903.02128v1

\bibitem{bgrt} Basabe, I. Gonzalez, J. Rudyak, Y.  Tamaki, D. Higher topological complexity and its symmetrization.
Algebr. Geom. Topol. 14(4) (2014), 2103–2124. 

\bibitem{ggl}
Cadavid-Aguilar, N. González, J. Gutierrez, D. Guzmán-Sáenz, A. Lara, A.
Sequential motion planning algorithms in real projective spaces: an approach to their immersion dimension. 
Forum Math. 30 (2018), no. 2, 397–417.

\bibitem{cv}
Cohen, D. Vandembroucq, L.
Motion planning in connected sums of real projective spaces. 
Topology Proc. 54 (2019), 323–334.

\bibitem{colman} Colman, H. Equivariant LS-category for finite group actions. 
Lusternik-Schnirelmann category and related topics.  
Contemp. Math., 316, 35-40, Amer. Math. Soc., Providence, RI, 2002.

\bibitem{davis} Davis, D. The symmetrized topological complexity of the circle. New York J. Math. 23 (2017), 593–602.

\bibitem{d} Dold, A. Lectures on algebraic topology. Classics in mathematics. Springer-Verlag, Berlin, 1995. 

\bibitem{farber} Farber, M. Invitation to Topological Robotics. Zurich Lectures in Advanced Mathematics. European Mathematical Society (EMS), Z\"urich, 2008.

\bibitem{fg} Farber, M. Grant, M. Symmetric motion planning. Topology and robotics, 85–104, Contemp. Math., 438, Amer. Math. Soc., Providence, RI, 2007.

 
\bibitem{jg} Gonzalez, J. 
Symmetric bi-skew maps and symmetrized motion planning in projective spaces.
Proc. Edinb. Math. Soc. (2) 61 (2018), no. 4, 1087–1100.


\bibitem{ggghmr}
Gonzalez, J. Gutierrez, B. Guzmam, A. Hidber, C. Mendoza, M. Roque, C.
Motion planning in tori revisited. Morfismos, 19(1) (2015), 7--18.

\bibitem{g} Grant, M. Symmetrized Topological Complexity.
J. Topol. Anal. 11 (2019), no. 2, 387–403.


\bibitem{gh} Grant, M. Colman, H. Equivariant topological complexity.
Algebr. Geom. Topol. 12 (2012), no. 4, 2299–2316.



\bibitem{m} Marzantowicz, W. 
A G-Lusternik-Schnirelman category of space with an action of a compact Lie group.
Topology 28 (1989), no. 4, 403–412.


\bibitem{n} Nakaoka, M. 
Cohomology theory of a complex with a transformation of prime period
and its applications. J. Inst. Polytech. Osaka City Univ. Ser. A., 7:51–102, 1956.

\bibitem{r} Rudyak, Y. On higher analogs of topological complexity.
Topol. Appl. 157 (2010) 916–920.

\bibitem{td}
tom Dieck, T.
Transformation groups and representation theory.
Lecture Notes in Mathematics, 766. Springer, Berlin, 1979.

\bibitem{etg} Torres-Giese, E. Permutation products and topological complexity. In preparation.
\end{thebibliography}
\end{document}